\documentclass[12pt]{amsart}
\usepackage{amssymb,latexsym}
\usepackage{enumerate}
\usepackage[colorlinks=true]{hyperref}
\hypersetup{urlcolor=blue, citecolor=red, linkcolor= black}

\makeatletter
\@namedef{subjclassname@2010}{%
\textup{2010} Mathematics Subject Classification}
\makeatother

\newtheorem{thm}{Theorem}[section]
\newtheorem{cor}[thm]{Corollary}

\newtheorem{prop}[thm]{Proposition}

\theoremstyle{definition}

\def \varpi {\bar \omega}

\def \R {\mathbb R}

\def \H {\mathbb H}

\def \C {\mathbb{C}}
\def\norm {\mid\!\mid}

\def\today{
\space\ifcase\month\or January\or February \or
March \or April\or May \or June\or July\or August\or
September\or October\or November\or
December\fi\space
\number\day ,\space\number\year} \nonstopmode

\usepackage{color}

\begin{document}


\baselineskip=17pt



\title[On quaternionic functions: I.] 
{On quaternionic functions: I.\ Local theory}

\author[P. Dolbeault]{Pierre Dolbeault}
\address{Institut de Math\'ematiques de Jussieu\\
UPMC, 4, place Jussieu 75005 Paris, France}
\email{\href{mailto:pierre.dolbeault@upmc.fr}{pierre.dolbeault@upmc.fr}}

\date{28.02.2014}

\begin{abstract}
Several sets of quaternionic functions are 
described and studied with respect to hyperholomorphy, addition and (non commutative) multiplication, on open sets of $\H$. The aim is to get a local function theory.
\end{abstract}



\maketitle
\thispagestyle{empty}
{\small\tableofcontents}

\section{Introduction}
We first recall the definition of the non commutative field $\H$ of quaternions using pairs of complex numbers and a modified Cauchy-Fueter operator (section 2) that have been introduced in \cite{CLSSS 07}. 
We will only use right multiplication; the (right) inverse of a nonzero quaternion being defined. We will consider $C^\infty$ $\H$-valued quaternionic functions defined on an open set $U$ of $\H$, whose behavior mimics the behavior of holomorphic functions near their zeroes on an open set of $\C$. If such a function does not identically vanish over $U$, it has an (algebraic) inverse which is defined almost everywhere on $U$. Finally, we describe properties of hyperholomorphic functions with respect to addition and multiplication. 

In section 3, we characterize the quaternionic functions which are, almost everywhere, hyperholomorphic and whose inverses are hyperholomorphic almost everywhere, on $U$, as the solutions of a system of two non linear PDE. We find non trivial examples of a solution, showing that the considered space of functions is significant: we will call these functions hypermeromorphic. 
There is a preliminary announcement in \cite{D13}

In section 4, we describe a subspace ${\mathcal H}_U$ of hyperholomorphic and hypermeromorphic functions defined almost everywhere on $U$, having ``good properties for addition and multiplication"; we obtain again systems of non linear PDE, and we give first results on spaces of functions of strictly positive dimension as vector spaces.

In the next paper II.\ Global theory, we will consider globalization of the introduced notions, define Hamilton 4-manifolds analogous to Riemann manifolds for $\H$ instead of $\C$, and give examples of such manifolds; our ultimate aim is to describe a class of four dimensional manifolds.

\section{Quaternions. $\H$-valued functions. Hyperholomorphic functions}
Quaternions, $\H$-valued functions have been defined in \cite{CLSSS 07}. 
\subsection {Quaternions} If $q\in\H$, then $q=z_1+z_2{\bf j}$ where $z_1, z_2\in\C$, 
hence $\H\cong\C^2\cong\R^4$ as complex or real vector space. We have:
$z_1{\bf j}={\bf j}\overline z_1$ (by computation in real coordinates); by definition, the {\it modulus} of $q$ is $\norm q\norm=({\vert z_1\vert} ^2+{\vert z_2\vert} ^2)^{\frac{1}{2}}$.

Let
* denote the {\it (right) multiplication} in $\H$:

Recall the noncommutativity of the multiplication: 

$qq'=(z_1+z_2{\bf j})*(z'_1+z'_2{\bf j})=(z_1z'_1-z_2\overline z'_2)+(z_1z'_2+z_2\overline z'_1){\bf j}=a+b{\bf j}$ 

$q'q= (z_1z'_1-z_2\overline z'_2)+(z'_1z_2+z'_2\overline z_1){\bf j}=a+ b'{\bf j}$, 

where $b=z_1z'_2+z_2\overline z'_1)$ and $b'=z'_1z_2+z'_2\overline z_1$. 

Commutativity when $q$ and $q' \in \R$.

The {\it conjugate} of $q$ is $\overline q= \overline z_1-z_2{\bf j}$. 
$q*\overline q=(z_1+z_2{\bf
j})*(\overline z_1-z_2{\bf j})=\vert z_1\vert ^2-z_1z_2{\bf j}+z_2{\bf
j}\overline z_1-z_2{\bf j}z_2{\bf j}=\vert z_1\vert ^2+\vert z_2\vert ^2=\vert q\vert$,
then: the {\it (right) inverse} of $q=z_1+z_2{\bf j}$ is: $q^{-1}=(\vert z_1\vert ^2+\vert
z_2\vert ^2)^{-1}\overline q=(\vert z_1\vert ^2+\vert
z_2\vert ^2)^{-1}(\overline z_1-z_2{\bf j})=\vert q\vert^{-1}\overline q$. Moreover: $(\vert z_1\vert ^2+\vert
z_2\vert ^2)^{-1}(\overline z_1-z_2{\bf j})*(z_1+z_2{\bf j})=1$, so the right inverse of $q^{-1}$ is $q$. 

\subsection {Quaternionic functions.}

Let $U$ be an open set of $\H\cong\C^2$ and $f\in C^\infty (U,\H)$, then 
$f=f_1+f_2{\bf j}$, where $f_1, f_2\in C^\infty (U,\C)$. The complex valued functions $f_1,f_2$ will be called the {\it components} of $f$.\smallskip

Remark that $\H$ is a real vector space in which real analysis is valid, in particular differential forms, distributions and currents are defined in $\H$. 
\smallskip
Remark that $\displaystyle\frac{\partial
f_1}{\partial\overline z_1}{\bf j}={\bf j}\frac{\partial\overline 
f_1}{\partial z_1}$ and analogous relations.

\subsection{Definitions} Let $U$ be an open neighborhood of 0 in $\H\cong\C^2$. 

\noindent
(a) {\it From now on, we will consider the quaternionic functions $f=f_1+f_2{\bf j}$ having the following properties}: 
\begin{enumerate}
\item[$(i)$] when $f_1$ and $f_2$ are not holomorphic, the set $Z(f_1)\cap Z(f_2)$ is {\it discrete} on $U$;
\item[$(ii)$] for every $q\in Z(f_1)\cap Z(f_2)$, $J_q^\alpha (.)$ denoting the {\it jet of order $\alpha$ at~$q$} \cite{M 66}, let $\displaystyle m_i=\sup_{\alpha_i} J_q^{\alpha_i}(f_i)=0$; $m_i$, $i=1,2$, is finite.
\end{enumerate}

{\it Define}: $\displaystyle m_q=\inf_i m_i$ as the {\it order of the zeroe $q$ of $f$}. 

\noindent
(b) We will also consider {\it the quaternionic functions defined almost everywhere on $U$} (i.e. outside of a subset of $U$ of Lebesgue measure 0, more precisely outside a finite set of $C^\infty$ hypersurfaces).

\subsection {Modified Cauchy-Fueter operator ${\mathcal D}$. 
Hyperholomorphic functions.} 
The modified Cauchy-Fueter operator ${\mathcal D}$ and hyperholomorphic functions have been defined in \cite{CLSSS 07,F 39}.      

For $f\in C^\infty (U,\H)$, with $f=f_1+f_2{\bf j}$, where $f_1, f_2\in C^\infty (U,\C)$, 
$$
{\mathcal D} f(q)=\frac{1}{2}\big(\frac{\partial}{\partial\overline z_1}
+{\bf j}\frac{\partial}{\partial\overline z_2}\big) f(q)=\frac{1}{2}\big({\frac{\partial
f_1}{\partial\overline z_1}-\frac{\partial\overline f_2}{\partial z_2}}\big)
(q)+{\bf j}\frac{1}{2}\big({\frac{\partial f_1}{\partial\overline
z_2}+\frac{\partial\overline f_2}{\partial z_1}}\big) (q).
$$

A function $f\in C^\infty (U,\H)$ is said to be {\it hyperholomorphic} if ${\mathcal
D}f=0$.

Characterization of the hyperholomorphic function $f$ on $U$:
\begin{equation}\label{Eq1}
{\frac{\partial
f_1}{\partial\overline z_1}-\frac{\partial\overline f_2}{\partial z_2}}=0; \ {\displaystyle\frac{\partial f_1}{\partial\overline
z_2}+\frac{\partial\overline f_2}{\partial z_1}}=0,\ {\rm on} \ U.
\end{equation}

The conditions: $f_1$ is holomorphic and $f_2$ is holomorphic are equivalent.
So {\it holomorphic functions with values in $\C$ will be identified with hyperholomorphic functions $f$ such that $f_2=0$.}

\begin{prop}The set ${\mathcal H}$ of hyperholomorphic functions such that the sum of two of them satisfies the above conditions $(i), (ii)$ is an $\H$-right vector space.
\end{prop}
\begin{proof}Let $f'=f'_1+f'_2{\bf j}$, $f''=f''_1+f''_2{\bf j}$ be two hyperholomorphic functions satisfying the above conditions $(i), (ii)$, then, for $\lambda', \lambda''\in {\H}, \lambda'f'+\lambda''f''$ has the same properties. \end{proof}

\begin{prop}The set ${\mathcal H}'$ of almost everywhere defined hyperholomorphic functions is an $\H$-right vector space.\end{prop}

\begin{prop}\label{Proposition 2.3.2.}Let $f'$, $f''$ be two almost everywhere defined hyperholomorphic functions. Then, their product $f'*f''$ satisfies:
$$
{\mathcal D} (f'*f'')={\mathcal D}f'*{\bf j}f''+\big(f'(\frac{\partial }{\partial\overline z_1})+\overline f'{\bf j}\frac{\partial }{\partial\overline z_2}\big)f''
$$
\end{prop}

\begin{proof}
$f'=f'_1+f'_2{\bf j}$, $f''=f''_1+f''_2{\bf j}$ be two hyperholomorphic functions. 
\medskip

We have: $f'*f''=(f'_1+f'_2{\bf j})(f''_1+f''_2{\bf j})
= f'_1f''_1- f'_2\overline f''_2+(f'_1f''_2+f'_2\overline f''_1){\bf j}$
\smallskip

\noindent Compute
$$
\frac{1}{2}\big(\frac{\partial}{\partial\overline z_1}
+{\bf j}\frac{\partial}{\partial\overline z_2}\big) \big(f'_1f''_1- f'_2\overline f''_2+(f'_1f''_2+f'_2\overline f''_1){\bf j}\big)
$$
By derivation of the first factors of the sum $f'*f''$, we get the first term: 
$$
\frac{1}{2}\big(\frac{\partial f'_1}{\partial\overline z_1}
+{\bf j}\frac{\partial f'_1}{\partial\overline z_2}\big)(f''_1+f''_2{\bf j}) +\frac{1}{2}\big(\frac{\partial f'_2}{\partial\overline z_1}
+{\bf j}\frac{\partial f'_2}{\partial\overline z_2}\big) {\bf j}{\bf j}(\overline f''_2-\overline f''_1{\bf j}) 
$$
$$
=\frac{1}{2}\big(\frac{\partial f'_1}{\partial\overline z_1}
+{\bf j}\frac{\partial f'_1}{\partial\overline z_2}\big)(f''_1+f''_2{\bf j}) +\frac{1}{2}\big(\frac{\partial f'_2{\bf j}}{\partial\overline z_1}
+{\bf j}\frac{\partial f'_2{\bf j}}{\partial\overline z_2}\big) {\bf j}(f''_2{\bf j}+f''_1)= {\mathcal D}f'*{\bf j}f''
$$
By derivation in 
$$
\frac{1}{2}\big(\frac{\partial}{\partial\overline z_1}
+{\bf j}\frac{\partial}{\partial\overline z_2}\big) \big(f'_1f''_1+f'_2{\bf j} f''_2{\bf j}+(f'_1f''_2{\bf j}+f'_2{\bf j}f''_1)\big)
$$
of the second factors of the sum $f'*f''$, we get the second term (up to factor~$\frac{1}{2}$):

\begin{multline*}
f'_1\frac{\partial f''_1}{\partial\overline z_1}+\overline f'_1{\bf j}\frac{\partial f''_1}{\partial\overline z_2}+f'_1\frac{\partial f''_2}{\partial\overline z_1}{\bf j}+\overline f'_1{\bf j}\frac{\partial f''_2}{\partial\overline z_2}{\bf j}
\\
+f'_2{\bf j}\frac{\partial f''_2}{\partial\overline z_1}{\bf j}+\overline f'_2{\bf j}\frac{\partial f''_2}{\partial\overline z_2}+f'_2{\bf j}\frac{\partial f''_1}{\partial\overline z_1}+\overline f'_2{\bf j}{\bf j}\frac{\partial f''_1}{\partial\overline z_2}
\\
=(f'_1+f'_2{\bf j})(\frac{\partial }{\partial\overline z_1})(f''_1+ f''_2{\bf j}) +(\overline f'_1+\overline f'_2{\bf j}){\bf j}\frac{\partial }{\partial\overline z_2}(f''_1+ f''_2{\bf j})
\\
= \big((f'_1+f'_2{\bf j})(\frac{\partial }{\partial\overline z_1})+(\overline f'_1+\overline f'_2{\bf j}){\bf j}\frac{\partial }{\partial\overline z_2}\big)(f''_1+ f''_2{\bf j})
\\
= \big(f'(\frac{\partial }{\partial\overline z_1})+\overline f'{\bf j}\frac{\partial }{\partial\overline z_2}\big)f''.
\end{multline*}
\end{proof}

If the components of $f'$ and $f''$ are real, the second term is: 
$$
\frac{1}{2}(f'_1+f'_2{\bf j})(\frac{\partial }{\partial\overline z_1}+{\bf j}\frac{\partial }{\partial\overline z_2})(f''_1+ f''_2{\bf j})=f'*
{\mathcal D}f''
$$
i.e.
\begin{cor}\label{Proposition 2.3.3} {\it The set ${\mathcal H}_{\R}$ of almost everywhere defined hyperholomorphic functions whose components are real is an $\R$-right algebra.} 
\end{cor} 

\subsection{Remark} The hyperholomorphic functions $f=f_1+f_2{\bf j}$ such that $f_2=0$ satisfy: 

$\displaystyle {\mathcal D} f(q)=\frac{1}{2}\big(\frac{\partial}{\partial\overline z_1}
+{\bf j}\frac{\partial}{\partial\overline z_2}\big) f(q)=\frac{1}{2}\frac{\partial
f_1}{\partial\overline z_1}(q)+{\bf j}\frac{1}{2}\frac{\partial f_1}{\partial\overline
z_2} (q)$ i.e. are holomorphic in $(z_1,z_2)$ with values in $\C$. 
\smallskip
More generally, we will consider {\it their product with a quaternion. Their set is the $\H$-right algebra of holomorphic functions.}

\section {Almost everywhere hyperholomorphic functions whose inverses are almost everywhere hyperholomorphic}

\subsection {Null set and inverse of a quaternionic function}

Let $f=f_1+f_2{\bf j}$ be a quaternionic function on $U$. 
The null set $Z(f)$ satisfies: $f_1=0; f_2=0$, then $Z(f)$ is of measure 0 in $U$. 

Ex.: $f_1=\overline z_1$; $f_2=\overline z_2$, then
$Z(f)=\{0\}$. 

Note that if $f$ is holomorphic, then, $f_2\equiv 0$
and $Z(f)$ is a complex hypersurface in $\C^2$.

We call {\it inverse} of a function $f: q\mapsto f(q)$, the function, defined almost everywhere on $U$: $f^{-1}: q\mapsto f(q)^{-1}$; then: $f^{-1}=\vert f\vert^{-1}\overline f$, where $\overline f$ is the (quaternionic) conjugate of $f$, and $f^{-1}=(\vert f_1\vert ^2+\vert f_2\vert ^2)^{-1}(\overline f_1-f_2{\bf j})$.

\subsection {Inversion and hyperholomorphy} 

{\it Assume $f$ to be hyperholomorphic and $Z(f)=\{0\}$, then
$\displaystyle\frac{1}{f}$ is not necessarily hyperholomorphic outside~$\{0\}$}. 

Ex.: $f=\overline z_1+\overline z_2{\bf j}$, then
$$
\frac{1}{f}=(z_1\overline z_1+z_2\overline z_2)^{-1}(z_1-\overline z_2{\bf j}); \hskip 2mm{\mathcal D}(\frac{1}{f})\not =0, 
$$

\subsubsection{Definition} Behavior of $f^{-1}$ at $q\in Z(f)=Z(f_1)\cap Z(f_2)$. Denoting $J_q^\alpha(.)$ the {\it jet of order $\alpha$ at $q$} \cite{M 66}, let $\displaystyle n_1= \sup J_q^\alpha (\vert f\vert \overline f_1^{-1})$; $n_2=\sup J_q^\alpha (\vert f\vert f_2^{-1})$ 

{\it Define}: $\displaystyle n_q=\sup_i n_i$, $i=1,2$ as the {\it order of the pole $q$ of $f^{-1}$}.

\subsubsection {Inverse of a holomorphic function}

Let $f= f_1+0{\bf j}$ be a hyperholomorphic function. Then $f^{-1} =f_1^{-1}+0{\bf j}$ and $f^{-1}$ is hyperholomorphic outside of the complex hypersurface $Z(f)$. Remark that $Z(f)$ is a subvariety of complex dimension 1, then of measure zero, in $U$.

We will consider {\it almost everywhere defined} hyperholomorphic functions on $U$. Ex.: holomorphic, meromorphic functions.

Let $f=f_1+f_2{\bf j}$ be a hyperholomorphic function and $g=g_1+g_2{\bf j}=\displaystyle \vert f\vert^{-1}(\overline f_1-f_2{\bf j})$ its inverse; so $g_1=\vert f\vert^{-1}\overline f_1$; $g_2=-\vert f\vert^{-1}f_2$, 
where $\vert f\vert=(f_1\overline f_1+f_2\overline f_2)$.
\subsubsection{A chacracterisation}

\begin{prop}\label{P31} The following conditions are equivalent
\begin{enumerate}
\item[$(i)$] the function $f=f_1+f_2{\bf j}$ and its right inverse are hyperholomorphic, when they are defined: 
\item[$(ii)$] we have the equations: 
$$
(\overline f_1- f_1)\frac{\partial\overline f_1}{\partial z_1}-\overline f_2\frac{\partial f_2}{\partial z_1}-f_2\frac{\partial\overline f_1}{\partial \overline z_2}=0, 
$$ 
$$
\overline f_2\frac{\partial f_1}{\partial z_1}+ \frac{\partial\overline f_2}{\partial z_1}(\overline f_1-f_1)-f_2\frac{\partial\overline f_2}{\partial\overline z_2}=0. 
$$
\end{enumerate}
\end{prop}

\begin{proof}
Let $f=f_1+f_2{\bf j}$ be a hyperholomorphic function and $g=g_1+g_2{\bf j}=\displaystyle \vert f\vert^{-1}(\overline f_1-f_2{\bf j})$ its inverse; so $g_1=\vert f\vert^{-1}\overline f_1$; $g_2=-\vert f\vert^{-1}f_2$, 
where $\vert f\vert=(f_1\overline f_1+f_2\overline f_2)$.
\begin{eqnarray*}
{\mathcal D}g(q)&=&\frac{1}{2}\big(\frac{\partial}{\partial\overline z_1}
+{\bf j}\frac{\partial}{\partial\overline z_2}\big) g(q)
=\frac{1}{2}\big({\frac{\partial
g_1}{\partial\overline z_1}-\frac{\partial\overline g_2}{\partial z_2}}\big)
(q)+{\bf j}\frac{1}{2}\big({\frac{\partial g_1}{\partial\overline
z_2}+\frac{\partial\overline g_2}{\partial z_1}}\big) (q)
\\
\frac{\partial g_1}{\partial\overline z_1}&=& \vert f\vert^{-1}\frac{\partial\overline f_1}{\partial \overline z_1}-\vert f\vert^{-2} \overline f_1\big(\frac{\partial f_1}{\partial \overline z_1}\overline f_1+f_1\frac{\partial\overline f_1}{\partial\overline z_1}+\frac{\partial f_2}{\partial \overline z_1}\overline f_2+f_2\frac{\partial\overline f_2}{\partial \overline z_1}\big)
\\
-\frac{\partial\overline g_2}{\partial z_2}&=& \vert f\vert^{-1}\frac{\partial\overline f_2}{\partial z_2}-\vert f\vert^{-2} \overline f_2\big(\frac{\partial f_1}{\partial z_2}\overline f_1+f_1\frac{\partial\overline f_1}{\partial z_2}+\frac{\partial f_2}{\partial z_2}\overline f_2+f_2\frac{\partial\overline f_2}{\partial z_2}\big)
\\
\frac{\partial g_1}{\partial\overline z_2}&=& \vert f\vert^{-1}\frac{\partial\overline f_1}{\partial \overline z_2}-\vert f\vert^{-2} \overline f_1\big(\frac{\partial f_1}{\partial \overline z_2}\overline f_1+f_1\frac{\partial\overline f_1}{\partial \overline z_2}+\frac{\partial f_2}{\partial \overline z_2}\overline f_2+f_2\frac{\partial\overline f_2}{\partial \overline z_2}\big)
\\
\frac{\partial\overline g_2}{\partial z_1}&=& - \vert f\vert^{-1}\frac{\partial\overline f_2}{\partial z_1}+\vert f\vert^{-2} \overline f_2\big(\frac{\partial f_1}{\partial z_1}\overline f_1+f_1\frac{\partial\overline f_1}{\partial z_1}+\frac{\partial f_2}{\partial z_1}\overline f_2+f_2\frac{\partial\overline f_2}{\partial z_1}\big)
\end{eqnarray*}
\begin{eqnarray*}
&&\hspace*{-36pt}2\vert f\vert^2{\mathcal D}g
\\
&=&(f_1\overline f_1+f_2\overline f_2)(\frac{\partial\overline f_1}{\partial \overline z_1}+\frac{\partial\overline f_2}{\partial z_2})-\overline f_1 f_1\frac{\partial \overline f_1}{\partial \overline z_1}-\overline f_1\overline f_1\frac{\partial f_1}{\partial \overline z_1}-\overline f_1f_2\frac{\partial\overline f_2}{\partial \overline z_1}-\overline f_1\overline f_2\frac{\partial f_2}{\partial \overline z_1}
\\
&&-\overline f_1\overline f_2\frac{\partial f_1}{\partial z_2}- f_1\overline f_2\frac{\partial\overline f_1}{\partial z_2}-
\overline f_2\overline f_2\frac{\partial f_2}{\partial z_2}- f_2\overline f_2\frac{\partial\overline f_2}{\partial z_2}
\\
&&+{\bf j}\Big((f_1\overline f_1+f_2\overline f_2)(\frac{\partial\overline f_1}{\partial \overline z_2}-\frac{\partial\overline f_2}{\partial z_1})-\overline f_1\overline f_1\frac{\partial f_1}{\partial \overline z_2}-\overline f_1 f_1\frac{\partial \overline f_1}{\partial \overline z_2}-\overline f_1\overline f_2\frac{\partial f_2}{\partial \overline z_2}-\overline f_1 f_2\frac{\partial\overline f_2}{\partial \overline z_2}
\\
&&+\overline f_1\overline f_2\frac{\partial f_1}{\partial z_1}+ f_1\overline f_2\frac{\partial\overline f_1}{\partial z_1}+
\overline f_2\overline f_2\frac{\partial f_2}{\partial z_1}+ f_2\overline f_2\frac{\partial\overline f_2}{\partial z_1}\Big)\end{eqnarray*}
Use the fact: $f$ is hyperholomorphic: 
$$
{\displaystyle\frac{\partial
f_1}{\partial\overline z_1}-\frac{\partial\overline f_2}{\partial z_2}}=0; \ {\displaystyle\frac{\partial f_1}{\partial\overline
z_2}+\frac{\partial\overline f_2}{\partial z_1}}=0 \leqno{\eqref{Eq1}}
$$ 
\begin{multline*}
2\vert f\vert^2{\mathcal D}g= f_1\overline f_1\frac{\partial\overline f_2}{\partial z_2} +f_2\overline f_2\frac{\partial\overline f_1}{\partial \overline z_1} -\overline f_1\overline f_1\frac{\partial f_1}{\partial \overline z_1}-\overline f_1f_2\frac{\partial\overline f_2}{\partial \overline z_1}-\overline f_1\overline f_2\frac{\partial f_1}{\partial z_2}-
\overline f_2\overline f_2\frac{\partial f_2}{\partial z_2}+\overline f_2\frac{\partial f_2}{\partial \overline z_1}(f_1-\overline f_1)+
\\+{\bf j}\Big(+f_2\overline f_2\frac{\partial\overline f_1}{\partial \overline z_2}-\overline f_1\overline f_2\frac{\partial f_2}{\partial \overline z_2}-\overline f_1 f_2\frac{\partial\overline f_2}{\partial \overline z_2}+\overline f_1\frac{\partial f_1}{\partial \overline z_2}(f_1-\overline f_1)
+\overline f_1\overline f_2\frac{\partial f_1}{\partial z_1}+ f_1\overline f_2\frac{\partial\overline f_1}{\partial z_1}+
\overline f_2\overline f_2\frac{\partial f_2}{\partial z_1}\Big)
\end{multline*}

$f$ being hyperholomorphic, $g$ hyperholomorphic is equivalent to the system of two equations:

$$
+f_1\overline f_1\frac{\partial\overline f_2}{\partial z_2} +f_2\overline f_2\frac{\partial\overline f_1}{\partial \overline z_1} -\overline f_1\overline f_1\frac{\partial f_1}{\partial \overline z_1}-\overline f_1f_2\frac{\partial\overline f_2}{\partial \overline z_1}-\overline f_1\overline f_2\frac{\partial f_1}{\partial z_2}-
\overline f_2\overline f_2\frac{\partial f_2}{\partial z_2}+\overline f_2\frac{\partial f_2}{\partial \overline z_1}(f_1-\overline f_1)=0
$$
$$
+f_2\overline f_2\frac{\partial\overline f_1}{\partial \overline z_2}-\overline f_1\overline f_2\frac{\partial f_2}{\partial \overline z_2}-\overline f_1 f_2\frac{\partial\overline f_2}{\partial \overline z_2}+\overline f_1\frac{\partial f_1}{\partial \overline z_2}(f_1-\overline f_1)
+\overline f_1\overline f_2\frac{\partial f_1}{\partial z_1}+ f_1\overline f_2\frac{\partial\overline f_1}{\partial z_1}+
\overline f_2\overline f_2\frac{\partial f_2}{\partial z_1}=0
$$

$f_1$ and $f_2$ satisfy, by conjugation of the second equation:
$$ +f_2\overline f_2\frac{\partial f_1}{\partial z_1} - f_1 f_1\frac{\partial\overline f_1}{\partial z_1}-f_1\overline f_2\frac{\partial f_2}{\partial z_1} + f_2\frac{\partial\overline f_2}{\partial z_1}(\overline f_1-f_1) +f_1\overline f_1\frac{\partial f_2}{\partial\overline z_2}-f_1f_2\frac{\partial \overline f_1}{\partial\overline z_2}-f_2f_2\frac{\partial\overline f_2}{\partial\overline z_2}=0
$$
$$
+\overline f_1\overline f_2\frac{\partial f_1}{\partial z_1}+ f_1\overline f_2\frac{\partial\overline f_1}{\partial z_1}.+
\overline f_2\overline f_2\frac{\partial f_2}{\partial z_1}+f_2\overline f_2\frac{\partial\overline f_1}{\partial \overline z_2}-\overline f_1\overline f_2\frac{\partial f_2}{\partial \overline z_2}-\overline f_1 f_2\frac{\partial\overline f_2}{\partial \overline z_2}+\overline f_1\frac{\partial f_1}{\partial \overline z_2}(f_1-\overline f_1)
=0
$$ 
Using $\eqref{Eq1}$, we get: 
$$
+f_2\overline f_2\frac{\partial f_1}{\partial z_1} + f_1 (\overline f_1- f_1)\frac{\partial\overline f_1}{\partial z_1}-f_1\overline f_2\frac{\partial f_2}{\partial z_1} + f_2\frac{\partial\overline f_2}{\partial z_1}(\overline f_1-f_1) -f_1f_2\frac{\partial \overline f_1}{\partial\overline z_2}-f_2f_2\frac{\partial\overline f_2}{\partial\overline z_2}=0
$$
$$
+\overline f_1\overline f_2\frac{\partial f_1}{\partial z_1}+ (f_1-\overline f_1)\overline f_2\frac{\partial\overline f_1}{\partial z_1}+
\overline f_2\overline f_2\frac{\partial f_2}{\partial z_1}+\overline f_1\frac{\partial f_1}{\partial \overline z_2}(f_1-\overline f_1)+f_2\overline f_2\frac{\partial\overline f_1}{\partial \overline z_2}.
-\overline f_1 f_2\frac{\partial\overline f_2}{\partial \overline z_2}=0
$$

Assume $f_1\not =0$, $f_2\not =0$

$$
\overline f_1 \big(f_2\overline f_2\frac{\partial f_1}{\partial z_1} + f_1 (\overline f_1- f_1)\frac{\partial\overline f_1}{\partial z_1}-f_1\overline f_2\frac{\partial f_2}{\partial z_1} + f_2\frac{\partial\overline f_2}{\partial z_1}(\overline f_1-f_1) -f_1f_2\frac{\partial \overline f_1}{\partial\overline z_2}-f_2f_2\frac{\partial\overline f_2}{\partial\overline z_2}\big)=0
$$
$$
-f_2\big(+\overline f_1\overline f_2\frac{\partial f_1}{\partial z_1}+ (f_1-\overline f_1)\overline f_2\frac{\partial\overline f_1}{\partial z_1}+
\overline f_2\overline f_2\frac{\partial f_2}{\partial z_1}
-\overline f_1\frac{\partial\overline f_2}{\partial z_1}(f_1-\overline f_1)+f_2\overline f_2\frac{\partial\overline f_1}{\partial \overline z_2}-\overline f_1 f_2\frac{\partial\overline f_2}{\partial\overline z_2}\big)=0
$$ 

By sum:
$$
\overline f_1 \big( f_1 (\overline f_1- f_1)\frac{\partial\overline f_1}{\partial z_1}-f_1\overline f_2\frac{\partial f_2}{\partial z_1} + f_2\frac{\partial\overline f_2}{\partial z_1}(\overline f_1-f_1) -f_1f_2\frac{\partial \overline f_1}{\partial\overline z_2}\big)
$$
$$
-f_2\big((f_1-\overline f_1)\overline f_2\frac{\partial\overline f_1}{\partial z_1}+
\overline f_2\overline f_2\frac{\partial f_2}{\partial z_1}
-\overline f_1\frac{\partial\overline f_2}{\partial z_1}(f_1-\overline f_1)+f_2\overline f_2\frac{\partial\overline f_1}{\partial \overline z_2}\big)=0
$$
i.e. 
$$
(\overline f_1f_1+f_2\overline f_2)\big ( (\overline f_1- f_1)\frac{\partial\overline f_1}{\partial z_1}-\overline f_2\frac{\partial f_2}{\partial z_1}-f_2\frac{\partial\overline f_1}{\partial \overline z_2}\big )=0
$$

$$
\overline f_2 \big(f_2\overline f_2\frac{\partial f_1}{\partial z_1} + f_1 (\overline f_1- f_1)\frac{\partial\overline f_1}{\partial z_1}-f_1\overline f_2\frac{\partial f_2}{\partial z_1} + f_2\frac{\partial\overline f_2}{\partial z_1}(\overline f_1-f_1) -f_1f_2\frac{\partial \overline f_1}{\partial\overline z_2}-f_2f_2\frac{\partial\overline f_2}{\partial\overline z_2}\big)=0
$$
$$
f_1\big(+\overline f_1\overline f_2\frac{\partial f_1}{\partial z_1}+ (f_1-\overline f_1)\overline f_2\frac{\partial\overline f_1}{\partial z_1}+
\overline f_2\overline f_2\frac{\partial f_2}{\partial z_1}
-\overline f_1\frac{\partial\overline f_2}{\partial z_1}(f_1-\overline f_1)+f_2\overline f_2\frac{\partial\overline f_1}{\partial \overline z_2}-\overline f_1 f_2\frac{\partial\overline f_2}{\partial\overline z_2}\big)=0
$$ 

By sum
$$
\overline f_2 \big(f_2\overline f_2\frac{\partial f_1}{\partial z_1} + f_2\frac{\partial\overline f_2}{\partial z_1}(\overline f_1-f_1) -f_2f_2\frac{\partial\overline f_2}{\partial\overline z_2}\big) 
$$
$$
+f_1\big(\overline f_1\overline f_2\frac{\partial f_1}{\partial z_1}
-\overline f_1\frac{\partial\overline f_2}{\partial z_1}(f_1-\overline f_1)-\overline f_1 f_2\frac{\partial\overline f_2}{\partial\overline z_2}\big)=0
$$ 
i.e.
$$
\overline f_2\frac{\partial f_1}{\partial z_1}+ \frac{\partial\overline f_2}{\partial z_1}(\overline f_1-f_1)-f_2\frac{\partial\overline f_2}{\partial\overline z_2}=0 
$$
\end{proof}
\begin{cor}\label{Cor32} {If $f$ satisfies the conditions of the Proposition, the same is true for $\alpha f$ with $\alpha\in \R$}.\end{cor}

\subsubsection{} Let $f=f_1+0{\bf j}$ be an almost everywhere holomorphic function, then the condition $(ii)$ of Proposition \ref{P31} is satisfied.

\begin{prop}\label{P33} Let $f=f_1+f_2{\bf j}$ be a quaternionic function such that $f_1$ and $f_2$ are real.

Then, the following conditions are equivalent:
\begin{enumerate}
\item[$(i)$] the function $f=f_1+f_2{\bf j}$ and its right inverse are hyperholomorphic, when they are defined; 
\item[$(ii)$] $f_1,f_2$ satisfy the equations:
$$
\frac{\partial f_2}{\partial z_1}+\frac{\partial f_1}{\partial \overline z_2}=0
$$ 
$$
\frac{\partial f_1}{\partial z_1}-\frac{\partial f_2}{\partial\overline z_2}=0
$$
\end{enumerate}
\end{prop}

\begin{proof}Assume $f_1$ and $f_2$ real: condition $(ii)$ of Proposition \ref{P31} reduces to: 
$$
f_2\frac{\partial f_2}{\partial z_1}+f_2\frac{\partial f_1}{\partial \overline z_2}=0;\hskip 3mm f_2\frac{\partial f_1}{\partial z_1}-f_2\frac{\partial f_2}{\partial\overline z_2}=0
$$
i.e.: $f_2=0$ and the linear system: 
$$
\frac{\partial f_2}{\partial z_1}+\frac{\partial f_1}{\partial \overline z_2}=0;
$$ 
$$
\frac{\partial f_1}{\partial z_1}-\frac{\partial f_2}{\partial\overline z_2}=0
$$

The solution $f_2=0$ means $f=f_1$=real; since $f_1$ is holomorphic and real, then $f_1=$ real constant: no interest.\end{proof}
\noindent {\it Numerical example}.
{\it Let $f_1=z_1+\overline z_1+z_2+\overline z_2+A$, $f_2= -z_1-\overline z_1+z_2+\overline z_2+ B$, $A, B\in\R$, then: $f$ and $f^{-1}$ outside the zero set of $f$, are hyperholomorphic}.\medskip

The {\it the null set of $f= f_1+f_2{\bf j}$, for $f_1,f_2$ as above, for $A=B=0$, is:
$$
z_1+\overline z_1+z_2+\overline z_2=0
$$
$$
-z_1-\overline z_1+z_2+\overline z_2=0
$$}
i.e, by difference and sum:
$$
\overline z_1+z_1=0; z_2+\overline z_2=0,
$$
i.e. 
$$
x_1=0; x_2=0
$$
in $\R^4$.

\subsection{Definition} We will call {\it w-hypermeromorphic function} (w- for {\it weak}) any almost everywhere defined hyperholomorphic function whose right inverse is hyperholomorphic almost everywhere. \smallskip 

Recall: the conjugate (in the sense of quaternions) of $f=f_1+f_2{\bf j}$ is $\overline f=\overline f_1-f_2{\bf j}$, and $f^{-1}=\vert f\vert^{-1}\overline f$.

\section{On the spaces of hypermeromorphic functions.}

\subsection{Sum of two w-hypermeromorphic functions.}
Let $f, g$ be two almost everywhere hyperholomorphic functions whose inverses have the same property, then: 

1) $f+g$ is almost everywhere hyperholomorphic; 

2) the inverse of $f+g$, $(f+g)^{-1}$ has to be almost everywhere hyperholomorphic. \smallskip 

Recall $\displaystyle {\mathcal D}=\frac{1}{2}\big(\frac{\partial}{\partial\overline z_1}
+{\bf j}\frac{\partial}{\partial\overline z_2}\big)$.

$(f+g)^{-1}=\vert f+g\vert^{-1}(\overline {f+g})$; $\vert f+g\vert^{-1}$, being real, is transparent with respect to ${\mathcal D}$ because $z_1{\bf j}={\bf j}\overline z_1$ for $z_1\in\C$;

$$
{\mathcal D}((f+g)^{-1})=-{\vert f+g\vert}^{-2}{\mathcal D}(\vert f+g)\vert)(\overline {f+g})+\vert f+g\vert ^ {-1}{\mathcal D}(\overline {f+g})
$$
We must have:
$$
{\vert f+g\vert}^2{\mathcal D}((f+g)^{-1})=-{\mathcal D}(\vert f+g)\vert)(\overline {f+g})+\vert f+g\vert {\mathcal D}(\overline {f+g})=0
$$

\begin{prop}If $f$ and $g$ are two w-hypermeromorphic functions, then the following conditions are equivalent:
\begin{enumerate}
\item[$(i)$] the sum $h=f+g$ is w-hypermeromorphic; 
\item[$(ii)$] $h$ satisfies the following PDE:
$$
-\big( \frac{\partial \vert h\vert}{\partial\overline z_1}+{\bf j}\frac{\partial\vert h\vert}{\partial\overline z_2}\big) (\overline h_1-h_2{\bf j})+\vert h\vert  {\big(\frac{\partial}{\partial\overline z_1}
+{\bf j}\frac{\partial}{\partial\overline z_2}\big)}(\overline h_1-h_2{\bf j})=0
$$
\end{enumerate}
\end{prop}

If $f$ and $g$ are holomorphic, the condition $(i)$ is trivially satisfied. 

\begin{proof}Explicit the condition: $${\vert h\vert}^2{\mathcal D}(h^{-1})=-{\mathcal D}(\vert h)\vert)(\overline {h})+\vert h\vert {\mathcal D}(\overline {h})=0$$; with $\overline h=\overline h_1-h_2{\bf j}$

$$
2{\mathcal D}\overline h=\big(\frac{\partial}{\partial\overline z_1}
+{\bf j}\frac{\partial}{\partial\overline z_2}\big)(\overline h_1-h_2{\bf j})= \frac{\partial\overline h_1}{\partial\overline z_1}+\frac{\partial \overline h_2}{\partial z_2}-\big(\frac{\partial h_2}{\partial\overline z_1}-\frac{\partial h_1}{\partial z_2}\big){\bf j}
$$

\begin{multline*}
{\mathcal D}(\vert h\vert)={\mathcal D}(h_1\overline h_1+h_2\overline h_2)=\frac{1}{2}\big(\frac{\partial}{\partial\overline z_1}
+{\bf j}\frac{\partial}{\partial\overline z_2}\big)(h_1\overline h_1+h_2\overline h_2)
\\
=
\frac{1}{2}\big(\overline h_1\frac{\partial h_1}{\partial\overline z_1} +\overline h_2\frac{\partial h_2}{\partial\overline z_1} +h_1\frac{\partial\overline h_1}{\partial\overline z_1} +h_2\frac{\partial\overline h_2}{\partial\overline z_1}\big)
\\
+\frac{1}{2}\big(\overline h_1\frac{\partial h_1}{\partial z_2} +\overline h_2\frac{\partial h_2}{\partial z_2}+h_1\frac{\partial\overline h_1}{\partial z_2} +h_2\frac{\partial\overline h_2}{\partial z_2}\big){\bf j}=0.
\end{multline*}\end{proof}

\noindent {\it Remark on functions of complex variables $z$,$\overline z$ with real values.} We assume such a function $f(z,\overline z)$ a convergente series in $(z,\overline z)$, $\displaystyle \sum_{0,0}^{\infty,\infty} a_{kl}z^k\overline z^l$. From $f=\overline f$, we get: $a_{kl}=\overline a_{lk}$. 
$\displaystyle\frac{\partial f}{\partial\overline z}=\sum la_{kl}z^k\overline z^{l-1}; \frac{\partial\overline f}{\partial z}=\sum k\overline a_{kl}\overline z^{k-1}z^l$; i.e. $\displaystyle\frac{\partial f}{\partial\overline z}=\frac{\partial f}{\partial z}$. The result remains valid if $f$ is $C^\infty$.

\begin{prop}If $f$ and $g$ are two w-hypermeromorphic functions whose components are real, then the following propeeties are equivalent:
\begin{enumerate}
\item[$(i)$] the sum $h=f+g$ is w-hypermeromorphic; 
\item[$(ii)$] $f$ and $g$ satisfy the following condition:
$h_1=A_1(z_2,\overline z_2); h_2=A_2(z_1,\overline z_1)$ and $A_1+A_2=B$, const.
\end{enumerate}
\end{prop}

Consider the full condition: $\vert h\vert^2{\mathcal D}(h^{-1})=-{\mathcal D}(\vert h\vert)\overline h+\vert h\vert {\mathcal D}\overline h=0$
\begin{proof}

From Proposition \ref{P33}, we have 
$$
\frac{\partial f_2}{\partial z_1}+\frac{\partial f_1}{\partial \overline z_2}=0 
$$ 
$$
\frac{\partial f_1}{\partial z_1}-\frac{\partial f_2}{\partial\overline z_2}=0
$$

Recall: $\displaystyle\frac{\partial 
f_1}{\partial\overline z_1}{\bf j}={\bf j}\frac{\partial\overline 
f_1}{\partial z_1}$

$$
2{\mathcal D} h=\big(\frac{\partial}{\partial\overline z_1}
+{\bf j}\frac{\partial}{\partial\overline z_2}\big)( h_1-h_2{\bf j})= \frac{\partial h_1}{\partial\overline z_1}+\frac{\partial h_2}{\partial z_2}-\big(\frac{\partial h_2}{\partial\overline z_1}-\frac{\partial h_1}{\partial z_2}\big){\bf j}
$$

$${\mathcal D}(\vert h\vert)={\mathcal D}(h_1^2+h_2^2)=\frac{1}{2}\big(\frac{\partial}{\partial\overline z_1}
+{\bf j}\frac{\partial}{\partial\overline z_2}\big)(h_1^2+h_2^2)=
h_1\frac{\partial h_1}{\partial\overline z_1} +h_2\frac{\partial h_2}{\partial\overline z_1} +\big(h_1\frac{\partial h_1}{\partial z_2} +h_2\frac{\partial h_2}{\partial z_2}\big){\bf j},
$$

$$
{\vert h\vert}^2{\mathcal D}(h^{-1})=-{\mathcal D}(\vert h)\vert)\overline h+\vert h\vert {\mathcal D}(\overline {h})=-\big(h_1\frac{\partial h_1}{\partial\overline z_1} +h_2\frac{\partial h_2}{\partial\overline z_1} +\big(h_1\frac{\partial h_1}{\partial z_2} +h_2\frac{\partial h_2}{\partial z_2}\big){\bf j}\big)(h_1-h_2{\bf j})
$$
$$
+(h_1^2+h_2^2)\frac{1}{2}\big(\frac{\partial h_1}{\partial\overline z_1}+\frac{\partial h_2}{\partial z_2}-\big(\frac{\partial h_2}{\partial\overline z_1}-\frac{\partial h_1}{\partial z_2}\big){\bf j}\big)=0
$$
\begin{equation}\label{13'}
\cong\big(-h_1^2\frac{\partial h_1}{\partial\overline z_1}+3h_2^2\frac{\partial h_2}{\partial z_2}\big)+\big(-h_1^2\frac{\partial h_2}{\partial z_2}+h_2^2\frac{\partial h_1}{\partial\overline z_1}\big){\bf j}=0 
\end{equation}
{}From \eqref{13'}
$$
-h_1^2\frac{\partial h_1}{\partial\overline z_1}+3h_2^2\frac{\partial h_2}{\partial z_2}=0; h_2^2\frac{\partial h_1}{\partial\overline z_1}-h_1^2\frac{\partial h_2}{\partial z_2}=0
$$
Then: either $h_1^4-3h_2^4=0,$ or:
$$
\frac{\partial h_1}{\partial\overline z_1}=0; \frac{\partial h_2}{\partial z_2}=0
$$
and, previously
$$
\frac{\partial h_1}{\partial z_2}=-\frac{\partial h_2}{\partial\overline z_1}
$$
Then: $h_1=A_1(z_2,\overline z_2); h_2=A_2(z_1,\overline z_1)$ and $A_1+A_2=B$, const. \end{proof}

\subsection{Product of two w-hypermeromorphic functions}

\begin{prop}Let $f$, $g$ be two w-hypermeromorphic functions on $U$, then the following conditions are equivalent:
\begin{enumerate}
\item[$(i)$] the product $f*g$ is w-hypermeremorphic;
\item[$(ii)$] $f$ and $g$ satisfy the system of PDE:

$$
g_1(\frac{\partial
f_1}{\partial\overline z_1}+\frac{\partial\overline f_2}{\partial z_2})+(f_1-\overline f_1)\frac{\partial g_1}{\partial\overline z_1}+\overline f_2\frac{\partial g_1}{\partial z_2}-f_2\frac{\partial\overline g_2}{\partial\overline z_1}=0
$$

$$
g_1(\frac{\partial
f_1}{\partial\overline z_2}-\frac{\partial\overline f_2}{\partial z_1})+(f_1-\overline f_1)\frac{\partial g_1}{\partial\overline z_2}-\overline f_2\frac{\partial g_1}{\partial z_1}-f_2\frac{\partial\overline g_2}{\partial\overline z_2}=0
$$
\end{enumerate}
\end{prop}

\begin{proof}
Let $f=f_1+f_2{\bf j}$ and $g=g_1+g_2{\bf j}$ two hypermeromorphic functions and $f*g=f_1g_1-f_2\overline g_2+(f_1g_2-f_2\overline g_1){\bf j}$ their product, then
\begin{eqnarray*}
&&\frac{\partial f_1}{\partial\overline z_1}-\frac{\partial\overline f_2}{\partial z_2}=0;
\\
&&\frac{\partial
(f_1g_1-f_2\overline g_2)}{\partial\overline z_1}-\frac{\partial(\overline f_1\overline g_2-\overline f_2g_1)}{\partial z_2}\\
&&\qquad=g_1(\frac{\partial
f_1}{\partial\overline z_1}+\frac{\partial\overline f_2}{\partial z_2})-\overline g_2(\frac{\partial\overline f_1}{\partial z_2}+\frac{\partial f_2}{\partial\overline z_1})+f_1\frac{\partial g_1}{\partial\overline z_1}-\overline f_1\frac{\partial\overline g_2}{\partial z_2}+\overline f_2\frac{\partial g_1}{\partial z_2}-f_2\frac{\partial\overline g_2}{\partial\overline z_1}=0
\\
&&g_1(\frac{\partial
f_1}{\partial\overline z_1}+\frac{\partial\overline f_2}{\partial z_2})+f_1\frac{\partial g_1}{\partial\overline z_1}-\overline f_1\frac{\partial\overline g_2}{\partial z_2}+\overline f_2\frac{\partial g_1}{\partial z_2}-f_2\frac{\partial\overline g_2}{\partial\overline z_1}=0.
\\
&&\frac{\partial
(f_1g_1-f_2\overline g_2)}{\partial\overline z_2}+\frac{\partial(\overline f_1\overline g_2-\overline f_2g_1)}{\partial z_1}\\&&\qquad =g_1(\frac{\partial
f_1}{\partial\overline z_2}-\frac{\partial\overline f_2}{\partial z_1})+\overline g_2(\frac{\partial\overline f_1}{\partial z_1}-\frac{\partial f_2}{\partial\overline z_2})+f_1\frac{\partial g_1}{\partial\overline z_2}-\overline f_1\frac{\partial\overline g_2}{\partial z_1}+\overline f_2\frac{\partial g_1}{\partial z_1}-f_2\frac{\partial\overline g_2}{\partial\overline z_2}=0
\\
&&g_1(\frac{\partial
f_1}{\partial\overline z_2}-\frac{\partial\overline f_2}{\partial z_1})+f_1\frac{\partial g_1}{\partial\overline z_2}+\overline f_1\frac{\partial\overline g_2}{\partial z_1}-\overline f_2\frac{\partial g_1}{\partial z_1}-f_2\frac{\partial\overline g_2}{\partial\overline z_2}=0
\end{eqnarray*}
\end{proof}

{\bf Remark.} If $f$ and $g$ are holomorphic, or meromorphic, the condition~$(ii)$ is trivially satisfied.

\begin{prop}Let $f$, $g$ be two w-hypermeromorphic functions on $U$, whose component are real, then the following conditions are equivalent:

\begin{enumerate}
\item[$(i)$] the product $f*g$ is w-hypermeremorphic;
\item[$(ii)$] $f$ and $g$ satisfy the PDE: 
$$
\frac{\partial
f_1}{\partial z_1}\frac{\partial g_2}{\partial z_2}+\frac{\partial
f_1}{\partial z_2}\frac{\partial g_2}{\partial z_1}=0
$$ 
\end{enumerate}
\end{prop}
\begin{proof} Assume that $f_1, f_2,g_1, g_2$ are real, then: 
\begin{equation}\label{dag}
\left\{\begin{array}{l}
\displaystyle g_1(\frac{\partial
f_1}{\partial\overline z_1}+\frac{\partial f_2}{\partial z_2})+f_2(\frac{\partial g_1}{\partial z_2}-\frac{\partial g_2}{\partial\overline z_1})=0
\\ \\
\displaystyle g_1(\frac{\partial
f_1}{\partial\overline z_2}-\frac{\partial f_2}{\partial z_1})- f_2(\frac{\partial g_1}{\partial z_1}+\frac{\partial g_2}{\partial\overline z_2})=0
\end{array}\right.
\end{equation}
$f$ being w-hypermeromorphic satisfies:
\begin{equation}\label{dagdag}
\left\{\begin{array}{l}
\displaystyle\frac{\partial f_2}{\partial z_1}+\frac{\partial f_1}{\partial \overline z_2}=0;
\\ \\
\displaystyle\frac{\partial f_1}{\partial z_1}-\frac{\partial f_2}{\partial\overline z_2}=0
\end{array}\right.
\end{equation}
the same for $g$. Using \eqref{dagdag}, we get, from \eqref{dag}
$$
\left\{\begin{array}{l}
\displaystyle g_1\frac{\partial f_1}{\partial z_1}+f_2\frac{\partial g_2}{\partial z_1}=0
\\ \\
\displaystyle g_1\frac{\partial
f_1}{\partial z_2}-f_2\frac{\partial g_2}{\partial z_2}=0
\end{array}\right.
$$
To get $f_2$, $g_1$ not both 0, , we need the relation:
$$
\frac{\partial
f_1}{\partial z_1}\frac{\partial g_2}{\partial z_2}+\frac{\partial
f_1}{\partial z_2}\frac{\partial g_2}{\partial z_1}=0
$$ 
\end{proof}

\subsection{Definition} We will call {\it hypermeromorphic} the w-hypermeromorphic functions whose sum and product are w-hypermeromorphic. Their space is nonempty, since the product $1$ of a w-hypermeromorphic and its inverse is w-hypermeromorphic; moreover this space contains the right $\H$-algebra of the meromorphic functions.

\subsection{Hypermeromorphic functions whose two components are real}

Let $f,g$ be two hypermeromorphic functions on $U$. 

From Proposition \ref{P33}, they satisfy: 
$$
\frac{\partial f_2}{\partial z_1}+\frac{\partial f_1}{\partial \overline z_2}=0; 
\frac{\partial f_1}{\partial z_1}-\frac{\partial f_2}{\partial\overline z_2}=0
$$
$$
\frac{\partial g_2}{\partial z_1}+\frac{\partial g_1}{\partial \overline z_2}=0; 
\frac{\partial g_1}{\partial z_1}-\frac{\partial g_2}{\partial\overline z_2}=0
$$
The properties: the sum $h=f+g$ is hypermeromorphic, is equivalent to:

$h_1=A_1(z_2,\overline z_2); h_2=A_2(z_1,\overline z_1)$ and $A_1+A_2=B$, const.;

the product $f*g$ is hypermeremorphic, is equivalent to:

$$
\frac{\partial f_1}{\partial z_1}\frac{\partial g_2}{\partial z_2}+\frac{\partial
f_1}{\partial z_2}\frac{\partial g_2}{\partial z_1}=0
$$ 

\begin{prop}Let $f$ and $g$ be quaternionic functions such that their components be real. The following conditions are equivalent:
\begin{enumerate}
\item[$(i)$] $f, g$, their sum and their product are w-hypermeromorphic;
\item[$(ii)$] $f, g$ satisfy the following relations:

$$
\frac{\partial f_2}{\partial z_1}+\frac{\partial f_1}{\partial z_2}=0; 
\frac{\partial f_1}{\partial z_1}-\frac{\partial f_2}{\partial z_2}=0
$$
$$
\frac{\partial g_2}{\partial z_1}+\frac{\partial g_1}{\partial z_2}=0; 
\frac{\partial g_1}{\partial z_1}-\frac{\partial g_2}{\partial z_2}=0
$$

$$
\frac{\partial f_1}{\partial z_1}\frac{\partial g_1}{\partial z_1}-\frac{\partial f_1}{\partial z_2}\frac{\partial g_1}{\partial z_2}=0 
$$
\end{enumerate}
$f_1+g_1=A_1(z_2,\overline z_2); f_2+g_2=A_2(z_1,\overline z_1)$ where $A_1$ and $A_2$ are given and satisfy $A_1+A_2=B$, const.
\end{prop}

Altogether we have $10$ variables and $8$ relations.

\begin{prop}The set of hypermeromorphic functions whose two components are real is an $\R$-right algebra.
\end{prop} 

\begin{prop} {The set ${\mathcal M}$ of hypermeromorphic functions on $U$ is a subalgebra of the algebra of quaternionic functions.}
\end{prop}

\begin{prop} The set ${\mathcal A}$ of hypermeromorphic functions on $U$ is a ``field" with only associativity of the multiplication. \end{prop}

\subsection{Meromorphic functions and Hypermeromorphic functions} 

\subsubsection{Motivation.}

Meromorphic functions satisfy all PDE satisfied by hypermeromorphic functions. 

Meromorphic functions have, classically, behavior on their domain of definition $U$ compatible with the behavior of hypermeromorphic functions on $U$. In the same way the sum and the product of a hypermeromorphic function and a meromorphic function is a hypermeromorphic function 

\subsubsection{Conclusion.} The $\H$-algebra and the field of meromorphic functions are substructures of the set of hypermeromorphic functions. 

In particular the union of the set of meromorphic functions and the set of hypermeromorphic functions whose components are real give a large number of examples of hypermeromorphic functions.

\subsection{Acknowledgement}Guy Roos pointed out to me that quaternions whose components are real are complex numbers.

\end{document}